\documentclass[a4paper,12pt,centertags]{amsart}
\usepackage[a4paper,centering]{geometry}
\usepackage[pdfdisplaydoctitle,colorlinks,urlcolor=blue,linkcolor=blue,citecolor=blue]{hyperref}
\usepackage{mathtools}
\usepackage{xspace}
\usepackage[utf8]{inputenc}
\usepackage{mathrsfs}
\usepackage{bbm}
\usepackage[T1]{fontenc}
\usepackage{mathscinet}
\newtheorem{theorem}{Theorem}[section]
\newtheorem{lemma}[theorem]{Lemma}
\newtheorem{proposition}[theorem]{Proposition}

\theoremstyle{definition}

\newtheorem{assumption}[theorem]{Assumption}
\theoremstyle{remark}
\newtheorem{remark}[theorem]{Remark}
\numberwithin{equation}{section}
\DeclareMathOperator{\Div}{div}

\DeclareMathOperator{\Span}{span}
\newcommand{\e}{\operatorname{e}}
\newcommand{\im}{\mathrm{i}}
\newcommand{\Z}{\mathbf{Z}}
\newcommand{\R}{\mathbf{R}}
\newcommand{\C}{\mathbf{C}}
\newcommand{\uno}{\mathbbm{1}}
\newcommand{\E}{\mathbb{E}}
\newcommand{\Prob}{\mathbb{P}}
\newcommand{\loc}{{\textrm{\tiny loc}}}
\newcommand{\scal}[1]{\langle #1 \rangle}
\newcommand{\eqdef}{\vcentcolon=}
\newcommand{\Torus}[1][2]{{\mathbb{T}_{#1}}}
\newcommand{\levelset}{\mathcal{L}}
\newcommand{\cov}{\mathcal{C}}
\newcommand{\memo}[1]{\ensuremath{\framebox{\tiny\textbf{\kern-2pt\textsf{#1}}\kern-2pt}}\xspace}
\newcounter{ConstantsCounter}
\newcommand{\mconst}{\stepcounter{ConstantsCounter}\ensuremath{c_\text{\theConstantsCounter}}}
\newcommand{\mcref}[1]{\ensuremath{c_\text{\ref{#1}}}}
\makeatletter
\newcommand{\mcdef}[1]{\stepcounter{ConstantsCounter}\protected@write\@auxout{}{    \string\newlabel{#1}{{\theConstantsCounter}{}{}{constant.\theConstantsCounter}{}}}\hypertarget{constant.\theConstantsCounter}{\mcref{#1}}}
\makeatother
\def\MRnum#1 #2\empty{#1}
\renewcommand{\MRhref}[2]{
  \href{http://www.ams.org/mathscinet-getitem?mr=#1}{#2}}
\renewcommand{\MR}[1]{
  \relax
  \ifhmode\unskip\space\fi
  \MRhref{\MRnum#1\empty}
  {\texttt{\Tiny[MR\MRnum#1\empty]}}
}
\hypersetup{%
  pdftitle={Hausdorff dimension of the level sets of some stochastic PDEs from fluid dynamics},
  pdfauthor={L. Baglioni, M. Romito},
  pdfcreator={M. Romito}
}
\begin{document}
  \title[Hausdorff dimension of level sets]
    {Hausdorff dimension of the level sets of some stochastic PDEs from fluid dynamics}
  \author[L. Baglioni]{Lorenzo Baglioni}
    \address{Dipartimento di Matematica, Università di Pisa, Largo Bruno Pontecorvo 5, I--56127 Pisa, Italia}
    \email{\href{mailto:baglioni@dm.unipi.it}{baglioni@dm.unipi.it}}
  \author[M. Romito]{Marco Romito}
    \address{Dipartimento di Matematica, Università di Pisa, Largo Bruno Pontecorvo 5, I--56127 Pisa, Italia}
    \email{\href{mailto:romito@dm.unipi.it}{romito@dm.unipi.it}}
    \urladdr{\url{http://www.dm.unipi.it/pages/romito}}
  \date{August 19, 2013}
  \begin{abstract}
    We determine with positive probability the Hausdorff dimension of
    the level sets of a class of Navier--Stokes $\alpha$--models at
    finite viscosity, forced by mildly rough Gaussian white noise. 
  \end{abstract}
\maketitle
\section{Introduction}

This paper address the problem of determining the Hausdorff dimension
of the level sets of the solutions of some stochastic PDEs from fluid
dynamics in two space dimensions. Consider the following stochastic PDE,
\[
  \dot\theta
      + \nu (-\Delta)^\alpha\theta
      + u\cdot\nabla\theta
     = \dot\eta,
        \qquad t\geq0,x\in\Torus,
\]
on $[-\pi,\pi]^2$, with periodic boundary conditions and zero spatial
mean, where $\nu$, $\alpha$ and $M$ are suitable parameters, $\dot\eta$
is Gaussian noise and the transport velocity is given by
$u = \nabla^\perp (-\Delta)^{-M}\theta$.
We prove almost sure upper bounds, as well as lower bounds with positive
probability, on the Hausdorff and packing dimension of the level
sets of the random field $\theta(t)$ at any positive time $t>0$.

Our equation \eqref{e:spde} belongs to the class of \emph{Navier--Stokes
$\alpha$--like models} (see for instance \cite{OlsTit2007} and
the reference therein) and, when $\alpha=1$, the 2D Navier--Stokes
equations in the vorticity formulation correspond to $M=1$, the
surface quasi--geostrophic equation corresponds to $M=\tfrac12$,
while $M=-1$ describes the large scale flows of a rotating shallow
fluid (see \cite{Fal2007a}). In this paper we will assume $\alpha>1$
and $M\geq1$. Unfortunately, our results do not cover the above
mentioned cases, although they are the motivating examples.

Our interest in level sets for equations from fluid dynamics is
inspired by the theory developed for the $0$--level set
at a physical and numerical level \cite{Fal2007a}. Our technique
is not powerful enough to be able to say something about the
``physical case'', for a series of reasons: we work at finite
viscosity and not in the vanishing viscosity regime, we assume
hyper--dissipation, due to the limits of the techniques we employ,
we are not able to capture all of the interesting cases of
the transport velocity, we are dealing with rough noise.
On the other hand, this paper is a preliminary work and 
extensions are currently being developed.

The results presented here focus on any level set. On the other
hand the zero spatial mean condition gives a privileged status
to the zero level set. It would be expected to obtain stronger
results on the zero level set than on any other level set.
We conjecture (see Remark \ref{r:failed} that indeed the zero
level set should have a ``deterministic'' (that is, determined
almost surely) dimension.

We give a few details on the techniques used to achieve our results.
As in \cite{DalKhoNua2007,DalKhoNua2009}, the non--linear problem
is reduced to the linear problem (namely, the same equation without
the non--linearity or, in other words, the linearization at $0$)
by means of an absolute continuity result between the laws of the
two processes. In the above references the equivalence is provided
by the Girsanov theorem. In our problem Girsanov's transformation
cannot be applied (see Remark \ref{r:girsanov}), and we apply a weaker
result from \cite{DapDeb2004}, using the polynomial moments from
\cite{EssSta2010}. On the one hand this gives equivalence of the
laws at the level of the single time rather than of the full path,
but on the other hand this is enough for our purpose. The equivalence
result is already known from \cite{MatSui2005,Wat2010} and ours is
an alternative proof.

Once the problem is reduced to a linear equation, it
becomes more amenable and one can use the theory developed for
Gaussian processes \cite{Kah1985,Xia1995,Xia1997,WuXia2006} to
show almost sure upper bounds on the Hausdorff and packing dimension,
as well as lower bounds with positive probability.
\subsection{Notations}

Let $\Torus = [-\pi,\pi]^2$ be the $2$--dimensional torus. For every
$\gamma\in\R$ denote by $H^\gamma_\#(\Torus)$ the Sobolev space of
periodic functions on $\Torus$ with mean zero on $\Torus$, defined in
terms of the complex Fourier coefficients with respect to the Fourier basis
$\{\e^{\im k\cdot x} : k\in\Z^2\}$, as
\[
  H^\gamma_\#(\Torus)
    = \bigl\{(\xi_k)_{k\in\Z^d}\subset\C:
        \xi_0=0,\quad
        \|\xi\|_\gamma^2
    \eqdef\sum_{k\in\Z^d}|k|^{2\gamma}|\xi_k|^2<\infty\bigr\},
\]
with norm $\|\cdot\|_\gamma$. In particular, when $\gamma=0$, we use
the standard notation $L^2_\#(\Torus)$.
Define moreover the spaces of divergence--free vector fields $V_\gamma$ as
\[
  V_\gamma
    = \bigl\{(u_k)_{k\in\Z^2_\star}\subset \C^2:
        k\cdot u_k = 0\text{ for all }k,\ 
        u^i = (u_k^i)_{k\in\Z^d}\in H^\gamma_\#\text{ for i=1,2}\bigr\},
\]
where $\Z^2_\star = \Z^2\setminus\{(0,0)\}$, with norm
$\|u\|_\gamma^2 \eqdef \|u^1\|_\gamma^2 + \|u^2\|_\gamma^2$,
and $u = (u^1,u^2)$. In particular, set $H = V_0$.

Denote by $A$ the realization of the Laplace operator $-\Delta$ on
$L^2_\#(\Torus)$ with periodic boundary conditions. A real orthonormal
basis of $L^2_\#(\Torus)$ (and hence of each $H_\#^\gamma$) of
eigenvectors of $A$ is given as follows. Set
\[
  \Z^2_+
    = \{k\in\Z^2:k_2>0\} 
      \cup \{k\in\Z^2:k_1>0,k_2=0\},
\]
$\Z^2_- = -\Z^2_+$ and $\Z^2_\star = \Z^2_+\cup\Z^2_-$, and
$e_k = \mcref{cc:onb}\sin k\cdot x$ for $k\in\Z^2_+$ and
$e_k = \mcref{cc:onb}\cos k\cdot x$ for $k\in\Z^d_-$,
where $\mcdef{cc:onb} = \sqrt{2}(2\pi)^{-1}$, then \cite{BerMaj2002}
$(e_k)_{k\in\Z^2_\star}$ is an orthonormal basis of $L^2_\#(\Torus)$.
With these positions, given $\gamma\in\R$,
$A^\gamma x = \sum_{k\in\Z^2_\star} |k|^{2\gamma} x_k e_k$
for $x = \sum_k x_k e_k$.
Set $E_k = \tfrac{k^\perp}{|k|}e_k$ for
every $k\in\Z^2_\star$, then $(E_k)_{k\in\Z^2_\star}$ is
an orthonormal basis of $H$. With a slight abuse of notations,
we will also denote by $A$ the realization of the Laplace operator
on $H$.

Given $y\in\R$ and a field $v:\Torus\to\R$, denote by $\levelset_y(v)$
the $y$--level set of $v$, namely $\levelset_y(v) = \{x\in\Torus: v(x) = y\}$.

We denote by $\dim_H$ and by $\dim_P$ the Hausdorff and the packing dimension,
respectively. We refer to \cite{Fal2003} for their definition and
properties.
\section{Formulation of the problem and main results}

\subsection{Formulation of the problem}

Fix $\nu>0$, $\alpha\geq1$ and $M\in\R$. Consider on $\Torus$ the following
stochastic PDE,
\begin{equation}\label{e:spde}
  \dot\theta
      + u\cdot\nabla\theta
    = - \nu A^\alpha\theta
      + \dot\eta,
        \qquad t\geq0,x\in\Torus,
\end{equation}
with periodic boundary conditions and with $\iint_{\Torus}\theta\,dx = 0$,
where $u$ is given as $u = \nabla^\perp A^{-M}\theta$.
By its definition it turns out that
$\Div(u)=0$ and $\iint u(x)\,dx = 0$.  The 2D Navier--Stokes equations
in the vorticity formulation correspond to $M=1$ \cite{BerMaj2002},
the surface quasi--geostrophic to $M=\tfrac12$, while $M=-1$ describes
the large scale flows of a rotating shallow fluid (see \cite{Fal2007a}).
Unfortunately, none of these values can be covered by our results.
\subsubsection*{The non--linearity}

Set for a vector function $v$, with $\Div(v) = 0$, and a scalar
function $f$,
\[
  B(v, f)
    = v\cdot\nabla f
    = \Div(vf).
\]
We use the same notation $B(v,v')$ when $v$ is a vector and the
operator $B$ is understood component--wise, namely
$[B(v,v')]_i = B(v_i,v')$. Set moreover
\[
  B_M(x)
    = B(\nabla^\perp A^{-M} x, x)
\]
We stress two important properties of the non--linear term.
Their proofs are standard using integration by parts arguments.

\begin{lemma}
  For every $x$ and $v$, with $\Div v = 0$,
  \begin{equation}\label{e:B1}
    \scal{x, B(v,x)}_{L^2}
      = 0,
        \qquad\qquad
    \scal{A^{-M}x, B_M(x)}_{L^2}
      = 0.
  \end{equation}
\end{lemma}
\subsubsection*{The random forcing term}

The random forcing term $\dot\eta$ is modeled as a coloured in space
and white in time Gaussian noise, namely $\eta$ is a Wiener process with
covariance $\cov\in\mathscr{L}(L^2_\#(\Torus))$.
For our purposes, we will assume that $\cov$ has a smoothing effect.
\begin{assumption}[on the covariance]\label{a:noise}
  The operator $\cov$ is positive linear bounded on $L^2_\#(\Torus)$.
  The driving noise is \emph{homogeneous} in space,
  hence $\cov$ has the same eigenvectors of the operator $A$.
  Under these assumptions, there are numbers $(\sigma_k)_{k\in\Z^2_\star}$
  such that for $x = \sum_k x_k e_k\in L^2_\#$,
  \[
    \cov x
      = \sum_{k\in\Z^d_\star}\sigma_k^2 x_k e_k.
  \]
  Assume additionally that there exists $\delta\in(1-\alpha,2-\alpha)$
  such that
  \begin{equation}\label{e:mainass}
    \frac{\mcdef{cc:noise1}}{|k|^\delta}
      \leq |\sigma_k|
      \leq \frac{\mcdef{cc:noise2}}{|k|^\delta}.
  \end{equation}
\end{assumption}
\subsubsection*{The abstract formulation}

In conclusion \eqref{e:spde} can be recast
in its abstract form as
\begin{equation}\label{e:spde_abs}
  d\theta
    + \bigl(\nu A^\alpha\theta + B(u,\theta)\bigr)\,dt
    = \cov^{\frac12}\,dW.
\end{equation}
\subsection{The linear problem}

Our first result gives upper and lower bounds for the dimension
of the level sets of the linear version of the problem under
examination, namely,
\begin{equation}\label{e:z}
  \begin{cases}
    dz + \nu A^\alpha z = \cov^{1/2}\,dW,\\
    z(0) = 0.
  \end{cases}
\end{equation}
\begin{theorem}\label{t:linear}
  Let $\alpha\geq1$ and let Assumption \ref{a:noise} be true.
  For every $t>0$ and $y\in\R$,
  \begin{itemize}
    \item $\Prob[\dim_P(\levelset_y(z_t))\leq 3-\alpha-\delta] = 1$,
    \item $\Prob[\dim_P(\levelset_y(z_t))=\dim_H(\levelset_y(z_t))= 3-\alpha-\delta]>0$.
  \end{itemize}
\end{theorem}
The proof of the theorem is based on well--known techniques
\cite{Kah1985,Xia1995,Xia1997} for the dimension of level sets.
\begin{remark}\label{r:failed}
  Theorem \ref{t:linear} above says that $\levelset_y(z_t)$ has
  dimension equal $3-\alpha-\delta$ with positive probability.
  It should not be expected that in general this may hold with
  probability one. Indeed, if $y=0$, the set $\levelset_y(z_t)$
  is empty with positive probability, since $z_t$ is continuous
  and defined on a compact set. On the other hand if $y=0$ then
  $\levelset_0(z_t)\neq\emptyset$ with probability one,
  since $z_t$ is non--zero and with zero average. We conjecture
  that $\dim_P(\levelset_0(z_t))=\dim_H(\levelset_0(z_t))
  = 3-\alpha-\delta$ with probability one.
  
  An effective way to prove that a random set has an almost sure
  dimension is to show that the set contains a \emph{limsup
  random fractal} \cite{KhoPerXia2000}. It is easy to construct
  a limsup random fractal contained into $\levelset_0(z_t)$
  by using the sets $\{|z_t(x_k^n)|\leq\epsilon\}$ as
  building blocks, where $(x_k^n)_{k}$ is a dyadic grid.
  Unfortunately, the correlation between distant blocks is too
  strong to apply \cite[Corollary 3.3]{KhoPerXia2000}.
  
  A different approach to prove the conjecture could be based
  on existence and regularity of the occupation density of
  $z_t$ at $0$. First, the random field $z_t$ has an
  occupation density $\ell$ due to \cite[Theorem 22.1]{GemHor1980}
  (see also \cite[Theorem 3]{Pit1978}). Moreover, it is not
  difficult to see that $z_t$ satisfies the property of
  \emph{local $2(\alpha+\delta-1)$--nondeterminism}
  (see \cite{MonPit1987} for the definition, and \cite{Xia2008,Xia2009}
  for a recent account) and so, by \cite[Theorem 4]{Pit1978}
  (or \cite[Theorem 26.1]{GemHor1980}), the occupation 
  density $\ell$ is H\"older continuous. By \cite[Theorem 1]{MonPit1987},
  in order to prove the ``probability one'' statement, it is
  sufficient to show that $\ell(0)>0$ with probability one.
  For instance this is true if there is a random constant
  $\mcdef{cc:fail} = \mcref{cc:fail}(\omega)>0$ such that
  \[
    \text{Leb}_{\Torus}(\{x\in\Torus:\ |z_t(x)|\leq\epsilon\})
      \geq\mcref{cc:fail}\epsilon.
  \]
  We have been not able to prove that $\Prob[\ell(0)>0]=1$.
\end{remark}
\subsection{The non--linear problem}

We extend the results on the dimension of level sets by means of
a absolute continuity result. The random perturbation we consider
is not ``strong'' enough (in terms of regularization) to apply
Girsanov's theorem (which is a standard method when dealing with
non--linear terms of order zero, see for instance
\cite{DalKhoNua2007,DalKhoNua2009}).
We use an absolute continuity theorem of \cite{DapDeb2004}
to translate the dimension results on the linear problem to the
non--linear problem.
We remark that another option to prove the absolute continuity
could be given by the idea in \cite{MatSui2005} (see also
\cite{MatSui2008,Wat2010}).
\begin{theorem}\label{t:main}
  Let $\nu>0$, $\alpha>1$, and $M\geq1$, and assume the
  covariance $\cov$ satisfies Assumption \ref{a:noise}.
  Let $\theta$ be the solution of problem \eqref{e:spde}
  with $\theta(0)\in L^2_\#$, then for every $y\in\R$ and $t>0$,
  \begin{itemize}
    \item $\Prob[\dim_P(\levelset_y(\theta_t))\leq 3-\alpha-\delta] = 1$,
    \item $\Prob[\dim_P(\levelset_y(\theta_t))=\dim_H(\levelset_y(\theta_t))= 3-\alpha-\delta]>0$.
  \end{itemize}
\end{theorem}
As we shall see in the course of the proof of the above result,
the same holds when $\theta$ is the stationary solution.
\section{Linear results}

In this section we prove Theorem \ref{t:linear}. Existence and uniqueness
for the solution $z$ of \eqref{e:z}, as well as of its invariant measure
and strong mixing are a standard matter, see \cite{DapZab1992}. In the
next lemma we summarize a few results concerning point--wise properties
of $z$ that we will need in the rest of the section.
\begin{lemma}\label{l:zprop1}
  Under the assumptions of Theorem \ref{t:linear}, for every $t>0$
  and $x\in\Torus$, $z(t,x)$ is a centred Gaussian random variable
  such that $\mcref{cc:noise1}\sigma_t^2\leq\E|z(t,x)|^2\leq
  \mcref{cc:noise2}\sigma_t^2$, with $\sigma_t>0$. Moreover,
  there is $g_t:\R\to\R$ such that
  \[
    \E[|z(t,x) - z(t,y)|^2]
      = g_t(x-y),
  \]
  and $\mcdef{cc:z1}|x|^{2(\alpha+\delta-1)}\leq g_t(x)\leq
  \mcdef{cc:z2}|x|^{2(\alpha+\delta-1)}$.
\end{lemma}
\begin{proof}
  We can write $z(t,x)$ as
  \[
    z(t,x)
      = \sum_{k\in\Z^2_\star}\Bigl(\sigma_k\int_0^t
          \e^{-\nu|k|^{2\alpha}(t-s)}\,d\beta_k(s)\Bigr)e_k(x),
  \]
  where $(\beta_k)_{k\in\Z^2_\star}$ are independent standard Brownian motions.
  Given $x\in\mathbb{T}_2$, $t>0$, the real valued random variable $z(t,x)$
  is Gaussian and,
  \[
    \operatorname{Var}(z(t,x))
      = \sum_{k\in\Z^2_\star} \frac{\sigma_k^2e_k(x)^2}{2\nu|k|^{2\alpha}}
          \bigl(1 - \e^{-2\nu|k|^{2\alpha}t}\bigr)
      \approx \sigma_t^2
      = \sum_{k\in\Z^2_\star} \frac{1 - \e^{-2\nu|k|^{2\alpha}t}}
          {4\nu|k|^{2(\alpha+\delta)}}.
  \]
  The expectation of the increments yields
  \[
    \E[|z(t,x) - z(t,y)|^2]
     = \sum_{k\in\Z^2_\star}\frac{\sigma_k^2}{2\nu|k|^{2\alpha}}
         \bigl(1 - \e^{-2\nu|k|^{2\alpha}t}\bigr)\sin^2\frac{k}{2}(x-y)
     = g_t(x-y).
  \]
  Using \eqref{e:mainass} and the fact that $(1 - \e^{-2\nu|k|^{2\alpha}t})$
  is bounded from above and below by constants independent of $k$ (but not $t$),
  we see that
  \[
    g_t(x)
      \approx\sum_{k\in\Z^2_\star}\frac1{|k|^{2\alpha+2\delta}}\sin^2\frac{k}{2}(x-y),
  \]
  hence $g_t(x)\approx|x|^{2(\alpha+\delta-1)}$ by Lemma \ref{l:sin}
  below.
\end{proof}
\begin{lemma}\label{l:sin}
  Let $\gamma>0$. Then for all $x\in\R^d$ with $|x|<1$,
  \[
    \sum_{k\in\Z^d_\star} \frac1{|k|^{d+\gamma}}\sin^2(k\cdot x)
      \sim h_\gamma(x),
  \]
  where $h_2(x) = - |x|^2\log|x|$ and $h_\gamma(x) = |x|^{\gamma\wedge2}$
  otherwise.
\end{lemma}
\begin{proof}
  For the upper bound,
  \[
    \sum_{k\in\Z^d_\star} \frac1{|k|^{d+\gamma}}\sin^2\bigl(k\cdot x\bigr)
      \leq \sum_{|k|\cdot|x|\leq1} \frac{|x|^2}{|k|^{d+\gamma-2}}
           + \sum_{|k|\cdot|x|\geq1} \frac1{|k|^{d+\gamma}}
      \leq \memo{s} + \memo{l}.
  \]
  The term \memo{l} can be easily estimated by comparison with an integral,
  yielding $\memo{l}\leq \mconst|x|^\gamma$. For \memo{s} we use the elementary
  result,
  \[
    \sum_{|k|\leq A} \frac1{|k|^\beta}
      \sim
        \begin{cases}
          A^{(d-\beta)_+} & \beta\neq d,\\
          \log A  & \beta=d,\\
        \end{cases}
        \qquad A\text{ large,}
  \]
  to obtain immediately that $\memo{s}\sim h_\gamma(x)$.

  For the lower bound, an elementary computation shows that
  there is $\mcdef{cc:gamma} = \mcref{cc:gamma}(\gamma)>0$ such that,
  \[
    \frac1{|k|^{d+\gamma}}
      = \mcref{cc:gamma}\int_0^\infty \e^{-t|k|^2}t^{\frac{d+\gamma-2}2}\,dt,
  \]
  hence
  \[
    \begin{aligned}
    \sum_{k\in\Z^d_\star} \frac1{|k|^{d+\gamma}}\sin^2\bigl(\tfrac12 k\cdot x\bigr)
      &= \mcref{cc:gamma} \int_0^\infty
           \Bigl(\sum_{k\in\Z^d_\star}\e^{-t|k|^2}\sin^2\bigl(\tfrac12 k\cdot x\bigr)\Bigr)
           t^{\frac{d+\gamma-2}2}\,dt\\
      &= \frac12 \mcref{cc:gamma} \int_0^\infty
           \Bigl(\sum_{k\in\Z^d_\star}\e^{-t|k|^2}(1 - \cos(k\cdot x)\Bigr)
           t^{\frac{d+\gamma-2}2}\,dt\\
      &\geq\frac12 \mcref{cc:gamma} \int_0^1
           \bigl(\phi(t,0) - \phi(t,x)\bigr)
           t^{\frac{d+\gamma-2}2}\,dt,
    \end{aligned}
  \]
  where
  \[
    \phi(t,x)
      = \sum_{k\in\Z^d_\star}\e^{-t|k|^2}\cos(k\cdot x).
  \]
  The $\phi$ is the fundamental solution of the heat equation with periodic
  boundary conditions and mean zero. In particular,
  $\phi(t,0) - \phi(t,x)\geq \mconst t^{-d/2}\bigl(1\wedge|x|^2/t\bigr)$,
  hence
  \[
    \sum_{k\in\Z^d_\star} \frac1{|k|^{d+\gamma}}\sin^2\bigl(\tfrac12 k\cdot x\bigr)
      \geq \int_0^1 \frac{\mconst}{t^{\frac{d}2}}\Bigl(1\wedge\frac{|x|^2}{t}\Bigr)
             t^{\frac{d+\gamma-2}2}\,dt
      \geq \mconst h_\gamma(x),
  \]
  by a direct computation.
\end{proof}
\begin{remark}\label{r:distance}
  If we replace the usual Euclidean distance in the statement of Lemma
  \ref{l:zprop1} and Lemma \ref{l:sin} with the ``torus distance'', namely
  $|x-y|_{\Torus} = \inf_{k\in\Z^2} |x-y+2\pi k|$, the conclusions
  of both lemmata still hold true.
\end{remark}
\begin{remark}
  If $d=1$, the above lemma admits a probabilistic proof, using a Fourier
  series expansion of the fractional Brownian motion. Indeed, by \cite{Igl2005}
  it follows, by simple computations that exploit the explicit form of the covariance
  function of the process, that
  \[
    \sum_{k=1}^\infty \frac1{k^{1+2H}}\sin^2(\tfrac12kt)
      \sim t^{2H}.
  \]
  We have not been able to find a similar proof in the multi--dimensional case.
\end{remark}
\subsection{The upper bound}

The following proposition contains the first part of Theorem \ref{t:linear}.
\begin{proposition}
  Under the assumptions of Theorem \ref{t:linear},
  \[
    \dim_H(\levelset_y(z_t))
      \leq \dim_P(\levelset_y(z_t))
      \leq 3 - \alpha - \gamma.
  \]
\end{proposition}
\begin{proof}
  We know by Lemma \ref{l:zprop1} and Gaussianity that $z_t$ is
  $\gamma$--H\"older continuous for every $\gamma<\alpha+\delta-1$
  and that the H\"older coefficient $L_\gamma$ has finite polynomial
  moments \cite[Theorem 1.4.1]{Kun1990}. Let $\gamma<\alpha+\delta-1$
  and consider a ball $B_\epsilon(x)$ in $\Torus$, then
  \[
    \Prob[B_\epsilon(x)\cap\levelset_y(z_t)]
      \leq \Prob[y\in z_t(B_\epsilon(x)),\ |z_t(x)-y|\geq\epsilon^{\gamma}]
        + \Prob[|z_t(x)-y|\leq\epsilon^\gamma].
  \]
  On the first event there is $x_y\in B_\epsilon(x)$ such that $z_t(x_y)=y$,
  hence for a $\gamma'$ such that $\gamma<\gamma'<\alpha+\delta-1$, we have that
  $\epsilon^\gamma\leq|z_t(x)-z_t(x_y)|\leq L_{\gamma'}\epsilon^{\gamma'}$,
  hence for $n$ large enough,
  \[
    \Prob[B_\epsilon(x)\cap\levelset_y(z_t)]
      \leq \Prob[L_{\gamma'}\geq\epsilon^{\gamma-\gamma'}]
        + \mconst\epsilon^\gamma
      \leq \epsilon^{n(\gamma'-\gamma)}\E[L_{\gamma'}^n]
        + \mconst\epsilon^\gamma
      \leq \mconst\epsilon^\gamma.
  \]
  Consider now a covering of $\Torus$ of $2^{2k}$ balls $B_k$ of radius $2^{-k}$,
  and let $N_k$ be the smallest number of balls of radius $2^{-k}$ covering
  $\levelset_y(z_t)$. Clearly, $N_k\leq\sum \uno_{\{B_k\cap\levelset_y\neq\emptyset\}}$,
  hence $\E[N_k]\leq \mcdef{cc:box}2^{2k-\gamma k}$. By the first Borel--Cantelli lemma,
  $N_k\leq \mcref{cc:box}2^{2k-\gamma'' k}$ for $k$ large enough, a.s., where $\gamma''<\gamma$.
  Therefore, $\dim_P\levelset_y(z_t)\leq 2-\gamma''$, and the upper bound
  follows by taking $\gamma''\uparrow\alpha+\delta-1$.
\end{proof}
\subsection{The lower bound}

We prove the second part of Theorem \ref{t:linear}. To this end we
first give an estimate of the two--points covariance.
\begin{lemma}\label{l:cov}
  Under the same assumptions of Theorem \ref{t:linear}, let
  $t>0$. Then there is $\mcdef{cc:cov}>0$ such that
  for every $x,x'\in\Torus$, with $x\neq x'$,
  $\det(q_{xx'})\geq\mcref{cc:cov}|x-x'|^{2(\alpha+\delta-1)}$,
  where $q_{xx'}$ is the covariance matrix of $(z_t(x),z_t(x'))$.
\end{lemma} 
\begin{proof}
  Given $t>0$, define the numbers $a_k(t)$ by
  \[
    \sigma_t(x)^2
      = \sum_{k\in\Z^2_\star} \frac{\sigma_k^2}{2\nu|k|^{2\alpha}}
          \bigl(1 - \e^{-2\nu|k|^{2\alpha}t}\bigr)e_k(x)^2
      = \sum_{k\in\Z^2_\star} a_k(t)^2 e_k(x)^2.
  \]
  If $x,x'\in\Torus$, define $\sigma_t(x,x')$ by
  \[
    \sigma_t(x,x')
      = \E[z_t(x)z_t(x')]
      = \sum_{k\in\Z^2_\star} a_k(t)^2 e_k(x)e_k(x').
  \]
  Then a few elementary computations (using the fact that
  $a_{-k} = a_k$ and the symmetries of $\Z^2_\star$)
  show that
  \[
    \det(q_{xx'})
      = \frac12\sum_{m,n\in\Z^2_\star}a_m(t)^2 a_n(t)^2
          \sin^2\bigl((m+n)\tfrac{x-x'}{2}\bigr)
  \]
  By re--arranging the sum, we finally obtain
  \[
    \det(q_{xx'})
      = \sum_{k\in\Z^2_\star} A_k\sin^2(k\cdot \tfrac{x-x'}{2}),
        \quad\text{where}\quad
    A_k
      = \frac12\sum_{m+n=k;m,n\in\Z^2_\star}a_n(t)^2a_m(t)^2.
  \]
  Since $a_k(t)\sim |k|^{-2(\alpha+\delta)}$ and $\alpha+\delta>1$,
  it is easy to see that $A_k\sim |k|^{-2(\alpha+\delta)}$ and
  \[
    \det(q_{xx'})
      \geq \sum_{k\in\Z^2_\star} \frac{\mconst}
        {|k|^{2 + 2(\alpha+\delta-1)}}
        \sin^2(k\cdot \tfrac{x-x'}{2})
      \geq \mconst|x-x'|^{2(\alpha+\delta-1)}
  \]
  where the last inequality follows from Lemma \ref{l:sin}.
\end{proof}
\begin{proposition}
  Under the assumptions of Theorem \ref{t:linear},
  \[
    \dim_P(\levelset_y(z_t))
      \geq\dim_H(\levelset_y(z_t))
      \geq 3 - \alpha - \delta,
  \]
  with positive probability.
\end{proposition}
\begin{proof}
  We use Frostman's ideas, see \cite{Kah1985}.
  To this end, given a non--negative measure $\mu$ on $\Torus$ and a
  number $\gamma>0$, define the $\gamma$--energy of $\mu$ as
  \[
    \|\mu\|_\gamma
      = \int_{\Torus}\int_{\Torus}
        \frac{\mu(dx)\mu(dx')}{|x-x'|^\gamma},
  \]
  where $|\cdot|$ denotes the distance on $\Torus$ (Remark
  \ref{r:distance}). We proceed as in \cite{WuXia2006} (see also
  \cite{Xia1995,Xia1997}) and define the measures
  $\mu_n = \sqrt{2\pi n}\exp({-\frac12n|z_t(x)-y|^2})\,dx$.
  For our purposes, it is sufficient to show that there are
  $\mcdef{cc:frost1},\mcdef{cc:frost2}$ (independent of $n$),
  such that $\E[\mu_n(\Torus)]\geq\mcref{cc:frost1}$,
  $\E[\mu_n(\Torus)^2]\leq\mcref{cc:frost2}$ and
  $\E[\|\mu_n\|_\gamma]<\infty$ for every $\gamma<3-\alpha-\delta$.
  Indeed, by these facts it follows that there is a sub--sequence
  converging to a measure $\mu$. Moreover, $\mu$ is non--zero
  with probability $\mcref{cc:frost1}^2\mcref{cc:frost2}^{-1}$
  (see \cite{Kah1985}). By continuity of $z_t$ (Lemma \ref{l:zprop1},
  it follows that $\mu$ has support in $\levelset_y(z_t)$ and hence
  Frostman's lemma \cite[Theorem 10.3.2]{Kah1985} yields
  $\Prob[\dim_H\levelset_y(z_t)\geq\gamma]\geq
  \mcref{cc:frost1}^2\mcref{cc:frost2}^{-1}$.
  We notice preliminarily that
  \[
    \sqrt{2\pi n}\e^{-\frac12n|z_t(x)-y|^2}
      = \int_\R\e^{-\frac1{2n}u^2+\im u(z_t(x)-y)}\,du.
  \]
  As in \cite{WuXia2006}, simple computations yield,
  \[
  \begin{multlined}
    \mu_n(\Torus)
      =\int_{\Torus}\int_\R \e^{-\frac1{2n}u^2-\im u y}
        \E[\e^{\im uz_t(x)}]\,du\,dx = \\
      =\int_{\Torus}\int_\R \e^{-\im u y}
        \e^{-\frac12(\frac1n+\sigma_t(x)^2)u^2}\,du\,dx
      =\int_{\Torus} \sqrt{\frac{2\pi}{\frac1n+\sigma_t(x)^2}}
        \e^{-\frac{y^2}{2( \frac1n+\sigma_t(x)^2 )}}\,dx\geq\\
      \geq (1+\mcref{cc:z2}^2\sigma_t^2)^{-\frac12}
        \e^{-\frac{y^2}{2\mcref{cc:z1}^2\sigma_t^2}}
      = \mcref{cc:frost1}.
  \end{multlined}
  \]
  With similar computations, involving this time two dimensional Gaussian
  random variables, we see that
  \[
  \begin{multlined}[.9\linewidth]
    \mu_n(dx)\mu_n(dx')
      = \Bigl(\iint_{\R^2}\e^{-\frac12u\cdot g_{xx'}\cdot u^T}
        \e^{-\im y(u_1+u_2)}\,du\Bigr)\,dx\,dx' = \\
      = \frac{2\pi}{\sqrt{\det g_{xx'}}}
        \e^{-\frac12(y,y)\cdot g_{xx'}^{-1}\cdot(y,y)^T}\,dx\,dx'
      \leq \frac {2\pi}{\sqrt{\det g_{xx'}}}\,dx\,dx',
  \end{multlined}
  \]
  where $g_{xx'}=\frac1n I + q_{xx'}$ and $q_{xx'}$ is the covariance
  matrix of $(z_t(x),z_t(x'))$. By Lemma \ref{l:cov},
  $\det(g_{xx'})\geq\det(q_{xx'})\geq\mcref{cc:cov}|x-x'|^{2(\alpha+\beta-1)}$,
  hence $\mu_n(dx)\mu_n(dx')\leq\mconst|x-x'|^{-(\alpha+\beta-1)}$
  and it is immediate to deduce that $\E[\mu_n(\Torus)^2]\leq\mcref{cc:frost2}$.
  Likewise, we deduce that $\E[\|\mu_n\|_\gamma]$ is bounded uniformly
  in $n$ if $\gamma<3-\alpha-\delta$. 
\end{proof}
\section{Non--linear results}

We turn to the proof of Theorem \ref{t:main}. Our strategy
is based on the idea that if $\theta$ is the solution of
\eqref{e:spde_abs} and $z$ of \eqref{e:z}, and if the laws of
$\theta(t)$ and $z(t)$ are equivalent measures, then
Theorem \ref{t:linear} immediately implies Theorem
\ref{t:main}.
\begin{remark}\label{r:girsanov}
  A standard way to prove absolute continuity of laws of solutions
  of stochastic PDEs is the Girsanov transformation. In our
  case, to apply Girsanov's transformation, the quantity
  \[
    \int_0^t \|\cov^{-1/2}B_M(\theta)\|_{L^2}^2\,ds
      <\infty,
        \qquad\Prob-{a.s.}
  \]
  should be finite, at the very least. This happens when $\alpha>2$
  (see \cite[Theorem 3]{MatSui2005}).
\end{remark}
The following theorem could be proved by means of the same method
in \cite[Theorem 2]{MatSui2005}, which is indeed the case
$\alpha>1$, $M=1$ (see also \cite{Wat2010}). Here we present
an alternative proof based on the method introduced in
\cite{DapDeb2004} and on the polynomial moments proved
in Lemma \ref{l:improved}.
\begin{theorem}\label{t:equiv_gt1}
  Let $\alpha>1$, $M\geq1$ and assume the covariance is as in
  Assumption \ref{a:noise}. Let $x\in L^2_\#(\Torus)$,
  $\theta(\cdot;x)$ be the solution of \eqref{e:spde} with initial
  condition $x$ and $z$ be the solution of \eqref{e:z} with
  initial condition $z(0) = 0$. Then for every $t>0$ the law of
  $\theta(t;x)$ is equivalent to the law of $z(t)$.
\end{theorem}
We first state some preliminary results that will be necessary for the
absolute continuity theorem given above. The first result ensures
existence and uniqueness for the solutions of \eqref{e:spde_abs}.
Its proof is quite standard and follows the lines of the proof of
\cite[Theorem 2.9]{Fla2008}.
\begin{lemma}\label{l:exist_uniq}
  Let $\alpha\geq1$ and $M\geq1$, and let Assumption~\ref{a:noise}
  be in force. Given a probability $\mu$ on $H^{-1}_\#(\Torus)$
  with all polynomial moments finite in $H^{-1}_\#(\Torus)$, there
  exists a unique (path--wise) solution $\theta$ of \eqref{e:spde_abs}
  with initial distribution $\mu$ such that
  $\theta\in C([0,\infty);H^{-1}_\#(\Torus))\cap
  L^2_\loc([0,\infty);L^2_\#(\Torus))$. Moreover for every $m\geq1$
  and $T>0$,
  \[
    \begin{gathered}
      \E\Bigl[\sup_{[0,T]}\|\theta(t)\|_{-M}^{2m}\Bigr]
          + \nu\E\Bigl[\int_0^T\|\theta(t)\|_{-M}^{2m-2}\|\theta(t)\|_{\alpha-M}^2\,dt\Bigr]
        <\infty,\\      
      \E\Bigl[\log\Bigl(1 + \sup_{[0,T]}\|\theta(t)\|_{-1}^2
          + \nu\int_0^T\|\theta(t)\|_{L^2}^2\,dt\Bigr)\Bigr]
        <\infty.
    \end{gathered}
  \]
  Denote by $\theta(\cdot;x)$ the solution with initial distribution
  concentrated at $x$. Then the process
  $(\theta(\cdot;x))_{x\in H^{-1}_\#(\Torus)}$
  is a Markov process and the associated transition semigroup is Feller
  in $H^{-1}_\#(\Torus)$.

  If additionally $\mu$ has second moment finite in $L^2_\#(\Torus)$,
  then for every $\gamma<\alpha+\delta-1$,
  \[
    \E\Bigl[\log\Bigl(1 + \sup_{[0,T]}\|\theta(t)\|_{L^2}^2
        + \nu\int_0^T\|\theta(t)\|_\gamma^2\,dt\Bigr)\Bigr]
      <\infty.
  \]
  Moreover, the process $(\theta(\cdot;x))_{x\in L^2_\#(\Torus)}$
  is Feller in $L^2_\#(\Torus)$.
\end{lemma}
The next preliminary ingredient is to prove that there exists an invariant
measure for problem~\eqref{e:spde_abs} which has all polynomial moments
finite in $L^2$ (and better). This is done following
(almost) \cite{EssSta2010}.
\begin{lemma}\label{l:improved}
  Let $\alpha\geq1$, $M\geq1$ and let Assumption~\ref{a:noise} be true.
  Then there exists an invariant measure $\mu$ for the transition
  semigroup associated to problem~\eqref{e:spde_abs}. Moreover,
  for every $\gamma<\delta+\alpha-1$ and $m\geq1$ there is a
  number $\mcdef{cc:improve}>0$ such that
  \begin{equation}\label{e:im_bound}
    \int \|x\|_{L^2}^{2m-2}\|x\|_\gamma^2\,\mu(dx)
      \leq \mcref{cc:improve}.
  \end{equation}
\end{lemma}
\begin{proof}
  Consider the Galerkin approximations of \eqref{e:spde_abs}
  \begin{equation}\label{e:spde_approx}
    d\theta_N + (\nu A^\alpha\theta_N + \pi_NB(u_N,\theta_N))\,dt
      = \pi_N\cov^{\frac12}dW,
  \end{equation}
  where $\pi_N$ is the projection onto $\Span[e_k: |k|\leq N]$
  and $u_N = \nabla^\perp A^{-M}\theta_N$.
  It is fairly standard (see \cite{Fla2008}) to prove that for every $N$
  the above system admits an invariant
  measure $\mu_N$. If we are able to prove \eqref{e:im_bound} for
  each $\mu_N$ with a constant $\mcref{cc:improve}$ independent
  of $N$, then the lemma is proved. Indeed, the uniform bound ensures
  tightness of $(\mu_N)_{N\geq1}$ and, consequently, of the laws of
  each of the stationary solution of \eqref{e:spde_approx}
  with initial condition $\mu_N$. The same methods of the previous
  lemma ensure that, up to a sub--sequence, there is a solution of
  \eqref{e:spde_abs} which is limit of stationary laws, hence
  stationary itself. Its marginal $\mu$ at fixed time turns out to
  be an invariant measure for \eqref{e:spde_abs} and a limit point
  of $(\mu_N)_{N\geq1}$. By semi--continuity $\mu$ 
  verifies \eqref{e:im_bound}.

  It remains to prove \eqref{e:im_bound} for the Galerkin system.
  Given $N\geq1$, let $\theta_N$ be the stationary solution of
  \eqref{e:spde_approx} with marginal law $\mu_N$. 

  \emph{Step 1: estimates for the linear part}.
  For every $\lambda>0$ consider the solution $z_{\lambda,N}$
  of the following problem,
  \[
      dz_{\lambda,N}
          + (\nu A^\alpha z_{\lambda,N}
          + \lambda z_{\lambda,N})\,dt
        = \pi_N\cov^\frac12 dW,
  \]
  with initial condition $z_{\lambda,N}(0) = 0$, and recall that
  $W(t) = \sum_{k\in\Z^2_\star} \beta_k e_k$, with
  $(\beta_k)_{k\in\Z^2_\star}$ independent standard Brownian motions.
  Set for every $T>0$, $a\in(0,\tfrac12)$,
  $\epsilon\in(0,\alpha+\delta-1)$ and $\beta\in(0,1)$
  such that $\alpha(1-2a(1-\beta))<(\alpha+\delta-1-\epsilon)$,
  \[
    M_{a,\epsilon,\beta}(T)^2
      = \sum_{k\in\Z^2_\star} \frac1{|k|^{4a\alpha(1-\beta)+2\delta-2\epsilon}}
          \Bigl(\sup_{0\leq s< t\leq T}\tfrac{|\beta_k(t) - \beta_k(s)|}{|t - s|^a}\Bigr)^2.
  \]
  From Proposition 2.1 and Corollary 2.2 of \cite{EssSta2010} we have that
  $\E[M_{a,\epsilon,\beta}(T)^{2m}]\leq\mconst T^{m(1-2a)}$
  and, $\Prob-a.s.$,  $\|z_{\lambda,N}(t)\|_\epsilon^2\leq
  \mcdef{cc:zlambda}\lambda^{-2a\beta}M_{a,\epsilon,\beta}(T)^2$.
  For the rest of the proof fix values of $a$ and $\beta$ as required
  above.

  \emph{Step 2: estimates for the non--linear part}.
  Set $\eta_{\lambda,N} = \theta_N - z_{\lambda,N}$, then
  $\eta_{\lambda,N}$ solves
  \[
    \dot\eta_{\lambda,N}
        + \nu A^{\alpha}\eta_{\lambda,N}
        + \pi_N B_M(\theta_N)
      = \lambda z_{\lambda,N},
  \]
  with initial condition $\eta_{\lambda,N}(0) = \theta_N(0)$.
  For every $m\geq1$,
  \[
    \begin{multlined}[.9\linewidth]
      \frac{d}{dt}(1 + \|\eta_{\lambda,N}\|_{L^2}^2)^m
        = 2m(1 + \|\eta_{\lambda,N}\|_{L^2}^2)^{m-1}
          \bigl(
            - \nu\|\eta_{\lambda,N}\|_\alpha^2 + \\
            - \scal{\eta_{\lambda,N}, B_M(\theta_N)}
            + \lambda\scal{z_{\lambda,N},\eta_{\lambda,N}}
          \bigr),
    \end{multlined}
  \]
  hence using \eqref{e:B1}, H\"older's inequality and Sobolev's embeddings,
  \[
    \scal{\eta_{\lambda,N}, B_M(\theta_N)}
      \leq \frac\nu4\|\eta_{\lambda,N}\|_\alpha^2
           + \mcdef{cc:ineq}\|z_{\lambda,N}\|_\epsilon^4
           + \mcref{cc:ineq}\|z_{\lambda,N}\|_\epsilon^2\|\eta_{\lambda,N}\|_{L^2}^2,
  \]
  where $\epsilon\in(0,\alpha+\delta-1)$ can be chosen arbitrarily
  small (and $\mcref{cc:ineq} = \mcref{cc:ineq}(\epsilon)$, although is
  independent from $N$). Young's inequality and the inequalities
  of the previous step yield,
  \[
    \begin{aligned}
      \lefteqn{\frac{d}{dt}(1 + \|\eta_{\lambda,N}\|_{L^2}^2)^m
          + \frac32\nu m(1 + \|\eta_{\lambda,N}\|_{L^2}^2)^{m-1}\|\eta_{\lambda,N}\|_\alpha^2\leq}\qquad\quad\\
      &\leq \mcdef{cc:young}m(1 + \|\eta_{\lambda,N}\|_{L^2}^2)^{m-1}
          (\lambda^4+\|z_{\lambda,N}\|_\epsilon^4
          + \|z_{\lambda,N}\|_\epsilon^2\|\eta_{\lambda,N}\|_{L^2}^2)\\
      &\leq \mcref{cc:young}m(1 + \|\eta_{\lambda,N}\|_{L^2}^2)^{m-1}
          (\lambda^4+\|z_{\lambda,N}\|_\epsilon^4
          + \mcref{cc:zlambda}\lambda^{-2a\beta}M_{a,\epsilon,\beta}(T)^2
          \|\eta_{\lambda,N}\|_{L^2}^2).
  \end{aligned}
  \]
  Consider $\omega\in\{M_{a,\epsilon,\beta}(T)\leq R\}$ and choose
  $\lambda = \lambda_R$ so that
  \[
    \mcref{cc:young}\mcref{cc:zlambda}R^2\lambda_R^{-2a\beta}
      = \frac{\nu}4,
  \]  
  then by using the Poincaré inequality and again Young's inequality,
  \[
    \frac{d}{dt}(1 + \|\eta_{\lambda,N}\|_{L^2}^2)^m
        + \nu m(1 + \|\eta_{\lambda,N}\|_{L^2}^2)^{m-1}
          \|\eta_{\lambda,N}\|_\alpha^2
      \leq \mcdef{cc:fine}(1 + \lambda_R^{4m}
        + \|z_{\lambda,N}\|_\epsilon^{4m}).
  \]
  Finally, by integrating in $[0,T]$, on the event
  $\{M_{a,\epsilon,\beta}(T)\leq R\}$,
  \begin{equation}\label{e:pathwise}
    \nu m\int_0^T \|\eta_{\lambda_R,N}\|_{L^2}^{2m-2}\|\eta_{\lambda_R,N}\|_\alpha^2\,dt
      \leq (1 + \|\theta(0)\|_{L^2}^2)^m
        + \mcref{cc:fine}\int_0^T(1 + \lambda_R^{4m} + \|z_{\lambda_R,N}\|_{L^2}^{4m})\,dt.
  \end{equation}

  \emph{Step 3: estimates in $\lambda_R$}.
  Define for every integer $R\geq1$ the events
  $A_R = \{R-1 < M_{a,\epsilon,\beta}(T)\leq R\}$.
  By Fernique's theorem (see for instance \cite{DapZab1992}),
  for every $q\geq2$ and every $\epsilon<\alpha+\delta-1$,
  \[
    \E[\|z_{\lambda_R,N}(t)\|_\epsilon^q]
      \leq \mconst\bigl(\E[\|z_{\lambda_R,N}\|_\epsilon^2]\bigr)^{\frac{q}2}
      \leq \mconst\Bigl(\sum_{k\in\Z^2_\star}\frac{|k|^{2\epsilon-2\delta}}
        {\lambda_R + \nu|k|^{2\alpha}}\Bigr)^{\frac{q}2}
      \leq \mconst\lambda_R^{-\frac{q(\alpha+\delta-1-\epsilon)}{2\alpha}},
  \]
  since $\lambda_R + \nu|k|^2\geq\lambda_R\vee(\nu|k|^2)$.
  By our choice of $\lambda_R$, $a$, and $\beta$, if $q\geq2$,
  \begin{equation}\label{e:stima1}
    \E\Bigl[\sum_{R=1}^\infty \uno_{A_R}\int_0^T
        \|z_{\lambda_R,N}(t)\|_\epsilon^q\,dt\Bigr]
      \leq \E\Bigl[\sum_{R=1}^\infty\int_0^T \|z_{\lambda_R,N}(t)\|_\epsilon^q\,dt\Bigr]
      \leq \mconst T.
  \end{equation}
  Likewise,
  \begin{equation}\label{e:stima2}
    \E\Bigl[\sum_{R=1}^\infty \uno_{A_R}\lambda_R^q\Bigr]
      \leq \mconst\E\Bigl[(1+M_{a,\beta}(T))^{\frac{q}{a\beta}}\Bigr]
      \leq \mconst (1 + T^{q\frac{1-2a}{2a\beta}}).
  \end{equation}
  
  Finally, recall \eqref{e:pathwise} and use \eqref{e:stima1}
  and \eqref{e:stima2} to obtain
  \begin{equation}\label{e:stima3}
    \E\Bigl[\sum_{R=1}^\infty\uno_{A_R}\int_0^T
        \|\eta_{\lambda_R,N}\|_{L^2}^{2m-2}\|
        \eta_{\lambda_R,N}\|_\alpha^2\,dt\Bigr]
      \leq \mconst(1 + \E_{\mu_N}[\|x\|_{L^2}^{2m}]
        + T + T^{2m\frac{1-2a}{a\beta}}).
  \end{equation}
  
  \emph{Step 4: conclusion.}
  By assumption $\E_{\mu_N}[\|x\|_{L^2}^{2m}]$ is uniformly bounded
  in $N$ for every $m\geq1$, due to the estimate in $H^{-M}_\#(\Torus)$.
  Fix $\gamma\in(0,\alpha+\delta-1)$, then by \eqref{e:stima1}
  and \eqref{e:stima3},
  \[
    \begin{aligned}
      \lefteqn{T\E_{\mu_N}[\|x\|_{L^2}^{2m-2}\|x\|_\gamma^2]=}\qquad&\\
        &= \E\Bigl[\sum_{R=1}^\infty\uno_{A_R}\int_0^T \|\theta_N(t)\|_{L^2}^{2m-2}\|\theta_N\|_\gamma^2\,dt\Bigr]\\
        &\leq \mconst\E\Bigl[\sum_{R=1}^\infty\uno_{A_R}
          \Bigl(\int_0^T \|z_{\lambda_R,N}(t)\|_\gamma^{2m}\,dt
          + \int_0^T\|\eta_{\lambda_R,N}(t)\|_{L^2}^{2m-2}
          \|\eta_{\lambda_R,N}(t)\|_\alpha^2\,dt\Bigr)\Bigr]\\
        &\leq \mcdef{cc:opt}\bigl(1 + T + T^{2m\frac{1-2a}{a\beta}}
          + \E_{\mu_N}[\|x\|_{L^2}^{2m}]\bigr).
    \end{aligned}
  \]
  The above inequality holds for all $T>0$, hence if we take
  $T = 2\mcref{cc:opt}$ and use the Poincaré inequality,
  we obtain \eqref{e:im_bound} for $\mu_N$ (with a constant
  uniform in $N$).
\end{proof}
Next, we show that problem \eqref{e:spde_abs} has a unique invariant
measure which is strongly mixing. Moreover, the strong Feller property
ensures that the law of $\theta(t)$ is equivalent to the law of the
invariant measure, for every $t>0$. This allows us to reduce
absolute continuity of laws at each time to absolute continuity
of the invariant measures. As a marginal remark, we notice that
the next lemma holds also when $\alpha=1$. The assumption $\alpha>1$
simplifies slightly the proof and it is what we need to prove the
theorem.
\begin{lemma}\label{l:uniq}
  Let $\alpha>1$, $M\geq1$ and let Assumption~\ref{a:noise} be
  true. Then the transition semigroup is strong Feller in
  $H^{-1}_\#(\Torus)$ and problem \eqref{e:spde_abs} admits
  a unique invariant measure which is strongly mixing.
\end{lemma}
\begin{proof}
  Uniqueness and strong mixing follow immediately from Doob's theorem
  (see \cite{DapZab1996}) if we prove that the transition
  semigroup is strong Feller and irreducible in $H^{-1}_\#(\Torus)$.
  The proof is very similar to \cite{FlaMas1995}, we give only a sketch
  and the key estimates.
  
  Preliminarily, we prove the strong Feller property in
  $H^\epsilon_\#(\Torus)$, with $\epsilon>0$ small. 
  Let $\chi:[0,\infty)\to\R$ be a smooth non--increasing
  function such that $\chi\equiv1$ on $[0,1]$ and
  $\chi\equiv0$ on $[2,\infty)$ and set $\chi_R(r)=\chi(r/R^2)$
  and $B^R_M(x) = \chi_R(\|x\|_\epsilon^2)B_M(x)$ for
  $\epsilon\in(0,\alpha+\delta-1)$. The cut--off problem is
  \[
    d\theta_R + (\nu A^{\alpha}\theta_R + B^R_M(\theta_R))\,dt
      = \cov^{\frac12}dW.
  \]
  It is easy to show, as in Lemma \ref{l:exist_uniq}, that the
  above equation has a unique solution in $C([0,\infty);H^\epsilon_\#)$
  for each initial condition in $H^\epsilon_\#(\Torus)$.
  Strong Feller in $H^\epsilon_\#$ follows from the Bismut--Elworthy--Li
  formula \cite{ElwLix1994}, which yields for every bounded 
  measurable $\varphi$ and every $x,h\in H^\epsilon_\#(\Torus)$,
  \[
    \begin{aligned}
      |P_t^R\varphi(x+h) - P_t^R\varphi(x)|
        &\leq \frac{\mconst}{\sqrt{t}}\|\varphi\|_\infty
           \E\Bigl[\Bigl(\int_0^t\|\cov^{-\frac12}
           D_h\theta_R(s;y)\|_{L^2}^2\,ds\Bigr)^{\frac12}\Bigr]\\
        &\leq\frac{\mconst}{\sqrt{t}}\|\varphi\|_\infty
           \E\Bigl[\int_0^t\|D_h\theta_R(s;y)\|_\delta^2\,ds\Bigr]^{\frac12},
    \end{aligned}
  \]
  where $(P_t^R)_{t\geq0}$ is the transition semigroup corresponding to
  $\theta_R$ and $D_h\theta_R(\cdot;y)$ is the Gateaux derivative (with
  respect to the initial condition) of $\theta_R$ in the direction $h$.
  Let $\xi(s;y) = D_h\theta_R(s;y)$, then it is sufficient to compute
  the ``energy estimate'' of $\xi$ in $L^2(\Torus)$, and the ``dissipative''
  term will provide the estimate we need. Clearly, the most troublesome term
  is the non--linearity, which is estimated as follows,
  \[
    \begin{aligned}
      \scal{DB_R(\theta_R)\xi, \xi}_{L^2}
        &\leq\mconst\chi_R'(\|\theta_R\|_\epsilon^2)
          \|\theta_R\|_\epsilon^3\|\xi\|_{L^2}\|\xi\|_\alpha 
          + \mconst\chi_R(\|\theta_R\|_\epsilon^2)
          \|\theta_R\|_\epsilon\|\xi\|_{L^2}\|\xi\|_\alpha\\
        &\leq \nu\|\xi\|_\alpha^2 + \mconst R^6\|\xi\|_{L^2}^2.
    \end{aligned}
  \]
    
  Next, we prove strong Feller in $H^{-1}_\#$, following an idea in
  \cite{Rom10a}. We know that if $\varphi:H^{-1}_\#(\Torus)\to\R$
  is bounded measurable, then $P_t\varphi\in C_b(H^\epsilon_\#)$,
  and we want to prove that $P_t\varphi\in C_b(H^{-1}_\#)$.
  To this end, let $x_n\to x$ in $H^{-1}_\#$, then it is sufficient
  to show that $\theta(t;x_n)\to\theta(t;x)$ {a.s.} in $H^\epsilon_\#$,
  since by the Markov property and the Lebesgue theorem,
  $P_t\varphi(x_n) = \E[P_{t/2}\varphi(\theta(t/2;x_n))]$ converges
  to $\E[P_{t/2}\varphi(\theta(t/2;x))] = P_t\varphi(x)$.

  Set $w_n(t) = \theta(t;x_n) - \theta(t;x) = \theta_n - \theta$,
  and choose $\epsilon$ small enough so that $\epsilon+1-\alpha\leq0$,
  then the energy inequality in $H^\epsilon_\#$ yields
  \[
    \begin{aligned}
    \frac{d}{dt}\|w_n\|_\epsilon^2
         + 2\nu\|w_n\|_{\alpha+\epsilon}^2
      &= - 2\scal{A^\epsilon w_n,B_M(\theta_n) - B_M(\theta)}\\
      &\leq \nu\|w_n\|_{\alpha+\epsilon}^2
         + \mcdef{cc:sf}(\|\theta_n\|_{L^2}^2 + \|\theta\|_{2-2M+\epsilon}^2)\|w_n\|_{L^2}^2,
    \end{aligned}
  \]
  hence by Gronwall's lemma, for every $s\leq t$,
  \[
    \|w_n(t)\|_\epsilon^2
      \leq \|w_n(s)\|_\epsilon^2\e^{\mcref{cc:sf}\int_s^t(\|\theta_n\|_{L^2}^2 + \|\theta\|_{2-2M+\epsilon}^2)}
      \leq \|w_n(s)\|_\epsilon^2\e^{\mcref{cc:sf}\int_0^t(\|\theta_n\|_{L^2}^2 + \|\theta\|_{2-2M+\epsilon}^2)}.
  \]
  Notice that the exponential term is $\Prob$--{a.~s.} finite by
  the bounds in Lemma \ref{l:exist_uniq}. Integrate the above inequality
  for $s\in[0,t]$ to get
  \[
    \|w_n(t)\|_{L^2}^2
      \leq \frac1{t}\Bigl(\int_0^t\|w_n(s)\|_{L^2}^2\,ds\Bigr)\e^{\int_0^t\|\theta\|_{L^2}^2\,ds}.
  \]
  So it is sufficient to prove that $\int_0^t\|w_n(s)\|_{L^2}^2\,ds\to0$.
  Use \eqref{e:B1}, H\"older's inequality and Sobolev's
  embeddings on the energy inequality in $H^{-1}_\#$ for $w_n$,
  to get
  \[
    \begin{aligned}
      \frac{d}{dt}\|w_n\|_{-1}^2 + 2\nu\|w_n\|_{\alpha-1}^2
        &= -\scal{w_n, B(\nabla^\perp A^{-M}\theta_n, w_n) + B(\nabla^\perp A^{-M}w_n, \theta)}_{-1}\\
        &\leq \nu\|w_n\|_{L^2}^2 
          + \mcdef{cc:gron}\|\theta\|_{L^2}^2\|w_n\|_{-1}^2.
    \end{aligned}
  \]
  The Gronwall inequality finally yields
  \[
    \|w_n(t)\|_{-1}^2 + \nu\int_0^t \|w_n(s)\|_{L^2}^2\,ds
      \leq \|x_n - x\|_{-1}\e^{\mcref{cc:gron}\int_0^t\|\theta\|_{L^2}^2\,ds},
  \]
  and the right hand side converges to $0$, $\Prob$--{a.~s.}, since
  $x_n\to x$ in $H^{-1}_\#$.
  
  Similar computations also yield irreducibility as in \cite{FlaMas1995}
  (see also \cite{Fer1997,Fer1999}).
\end{proof}
We have all elements to prove Theorem \ref{t:equiv_gt1}.
Since both problems \eqref{e:z} and \eqref{e:spde_abs} satisfy
the strong Feller property and have irreducible transition
probabilities, the law of each process at some time $t>0$ and
the corresponding invariant measure are equivalent measures.
It is then sufficient to show equivalence of the invariant
measures. To this end choose $\epsilon\in(0,1)$ such that
$\epsilon>2-\alpha-\delta$ and $\epsilon<\alpha-\delta$
(which is possible since $\alpha>1$),
and let $\theta_\epsilon=A^{-\epsilon/2}\theta$ and
$z_\epsilon=A^{-\epsilon/2}z$. The new process
$\theta_\epsilon$ solves
\begin{equation}\label{e:spdepsilon}
  d\theta_\epsilon
      + \nu A^\alpha\theta_\epsilon
      + B_{M,\epsilon}(\theta_\epsilon)
    = \cov_\epsilon^{\frac12}\,dW,
\end{equation}
where $B_{M,\epsilon}(x)=A^{-\epsilon/2}B_M(A^{\epsilon/2}x)$
and $\cov_\epsilon=A^{-\epsilon}\cov$, and similarly for
$z_\epsilon$. We will prove equivalence of the invariant measures
of $\theta_\epsilon$ and $z_\epsilon$ using
\cite[Theorem 3.4]{DapDeb2004} on $\theta_\epsilon$ and
$z_\epsilon$.
\begin{proof}[Proof of Theorem \ref{t:equiv_gt1}]
  Before checking that the assumptions of the aforementioned
  theorem are satisfied, we show how to deduce the equivalence
  of measures. Indeed, the theorem in \cite{DapDeb2004} shows
  the following formula
  \[
    (\lambda - L_\epsilon)^{-1}
      = (\lambda - N_\epsilon)^{-1}
        - (\lambda - N_\epsilon)^{-1}
        \scal{B_{M,\epsilon}, D(\lambda - L_\epsilon)^{-1}},
      \qquad\lambda>0,
  \]
  where $L_\epsilon$ and $N_\epsilon$ are the generators of
  the Markov semigroups associated to $z_\epsilon$ and
  $\theta_\epsilon$, respectively. To prove equivalence
  it is sufficient to show that for every bounded measurable
  function $f$ on $L^2_\#$, $(\lambda - L_\epsilon)^{-1}f=0$
  if and only if $(\lambda - N_\epsilon)^{-1}f=0$.
  If $(\lambda - L_\epsilon)^{-1}f=0$, the conclusion is
  immediate (it is already in \cite{DapDeb2004}). Assume
  that $(\lambda - N_\epsilon)^{-1}f = 0$ and set
  $\phi_\lambda = (\lambda - L_\epsilon)^{-1}f=0$. By
  the formula above,
  \[
    \phi_\lambda
        + (\lambda - N_\epsilon)^{-1}
        \scal{B_{M,\epsilon}, D\phi_\lambda}
      =0.
  \]
  Clearly, $\phi_\lambda\in D(L_\epsilon)$ and, by
  Proposition 2.4 of \cite{DapDeb2004}, $\phi_\lambda\in C^1_b$.
  Therefore $\phi_\lambda\in D(N_\epsilon)$ and
  $(\lambda - N_\epsilon)\phi_\lambda = (\lambda -
  L_\epsilon)\phi_\lambda - \scal{B_{M,\epsilon},D\phi_\lambda}$.
  By the formula above it follows that
  $(\lambda-L_\epsilon)\phi_\lambda=0$, that is $\phi_\lambda=0$.
  This prove the equivalence.

  We turn to the verification of the assumptions of Theorem 3.4
  of \cite{DapDeb2004}. The assumptions for the
  linear problem \cite[Hypothesis 2.1]{DapDeb2004} are standard
  and can be verified as in Section 4 of the mentioned paper. We
  only give a few details on the fourth hypothesis, namely that
  $\|\Lambda_t\|$ admits Laplace transform defined on $(-1,\infty)$,
  where $\Lambda_t = \cov_{\epsilon,t}^{-1/2}\e^{-tA^\alpha}$ and
  \[
    \cov_{\epsilon,t}
      = \int_0^t \e^{-sA^\alpha}\cov_\epsilon\e^{-s(A^\alpha)^*}\,ds.
  \]
  By our assumptions each $e_k$ (the sine--cosine orthonormal basis)
  is an eigenvector of $\Lambda_t$. Denote by $\lambda_{t,k}$ the
  corresponding eigenvalue. Elementary computations yield
  \[
    \lambda_{t,k}
      = \Bigl(|\sigma_k|^2\frac{1-\e^{-2t|k|^{2\alpha}}}
        {2|k|^{2(\alpha+\epsilon)}}\Bigr)^{-\frac12}\e^{-t|k|^{2\alpha}}
      \sim\frac{\mconst}{t^{\frac{\alpha+\delta+\epsilon}{2\alpha}}}
  \]
  for $t$ small. The existence of the Laplace transform follows, since
  $\|\Lambda_t\|=\sup_k\lambda_{t,k}$ and
  $\delta+\epsilon<2-\alpha\leq\alpha$.

  As it regards the assumptions for the non--linear problem,
  Lemma \ref{l:exist_uniq} and Lemma \ref{l:uniq} ensure that
  problem \eqref{e:spdepsilon} has a unique solution and generates
  a Feller semigroup in $L^2_\#$. Moreover the semi--group has
  a unique invariant measure $\mu_\epsilon$ which is strongly mixing.
  Let 
  \[
    B_{M,\epsilon,N}(x)
      = \frac{N}{N + \|x\|_{L^2}}\pi_N B_M(\pi_N x),
  \]
  where $\pi_N$ is the projection onto $\Span[e_k: |k|\leq N]$,
  then $B_{M,\epsilon,N}$ is Lipschitz--continuous in
  $L^2_\#$. Moreover, if $\theta_{\epsilon,N}$ is the solution of
  \eqref{e:spdepsilon} when $B_{M,\epsilon}$ is replaced
  by $B_{M,\epsilon,N}$, then $\theta_{\epsilon,N}(t)
  \to\theta_\epsilon(t)$ for all $t>0$ and all
  $x\in L^2_\#$, whenever $\theta_{\epsilon,N}(0)=\theta_\epsilon(0)$.
  It remains to prove the main assumption of \cite[Theorem 3.4]{DapDeb2004},
  namely that $B_{M,\epsilon,N}(x)\to B_{M,\epsilon}(x)$, $\mu$--{a.~s.}
  in $L^2_\#$, and that there is a $g\in L^2(L^2_\#,\mu)$ such that
  $\|B_{M,\epsilon,N}(x)\|_{L^2}\leq g(x)$, $\mu$--{a.~s.}.
  Let $\epsilon_M=0$ if $M>1$, and an arbitrary value in
  $\bigl(0,\epsilon-(2-\delta-\alpha)\bigr)$ if $M=1$, and
  let $g_\epsilon=\mcdef{cc:epsm}\|\cdot\|_{\epsilon+\epsilon_M}
  \|\cdot\|_1$. Lemma \ref{l:improved} and interpolation
  immediately imply that $g_\epsilon\in L^2(L^2_\#,\mu_\epsilon)$
  since,
  \[
    \E^{\mu_\epsilon}[g_\epsilon^2(x)]
      \leq \mcref{cc:epsm}\E^\mu
        [\|x\|_{L^2}^2\|x\|_{1-\epsilon+\epsilon_M}^2].
  \]
  Moreover, by choosing $\mcref{cc:epsm}$ large enough, we have
  that $\|B_{M,\epsilon,N}(x)\|_{L^2}\leq g_\epsilon(x)$.
  Indeed, by using the embedding of $L^\infty$ into $H^{1+\gamma}$,
  for $\gamma\leq\epsilon_M\vee 2(M-1)$, we have that
  \[
    \|B_M(x)\|_{L^2}
      \leq \mconst\|x\|_{\epsilon_M}\|x\|_1,
        \qquad
    \|B_M(x)\|_{-1}
      \leq \mconst\|x\|_{\epsilon_M}\|x\|_0,
  \]
  and hence $\|B_M(x)\|_{-\epsilon}\leq \mconst\|x\|_{\epsilon_M}
  \|x\|_{1-\epsilon}$.
\end{proof}

\end{document}